\documentclass{article}
\usepackage{amsmath,amsfonts,amssymb,amsthm,epsfig,epstopdf,titling,url,array}
\usepackage{eqnarray}
\usepackage{enumerate}
\usepackage[a4paper]{geometry}
\usepackage{amscd}
\usepackage{centernot}
\usepackage{authblk}
\usepackage{cleveref}
\newcommand{\co} {\mathbb{C}}

\newcommand{\iy} {\infty}
\newcommand{\N} {\mathbb{N}}
\newcommand{\no} {\mathbb{N}_0}

\newcommand{\Z} {\mathbb{Z}}

\newcommand{\lm}{\lambda}

\newcommand{\eps}{\epsilon}
\newcommand{\x}{\mathcal X}

\newcommand{\h}{\mathcal X}

\newcommand{\bx}{\mathcal B(\mathcal X)}

\newcommand{\norm}[1]{\left\Vert#1\right\Vert}

\newcommand{\set}[1]{\left\{#1\right\}}

\newcommand{\brc}[1]{\left(#1\right)}

\newcommand{\LZ}{\ell^{2}(\mathbb Z)}

\newcommand{\LNp}{\ell^{p}(\mathbb N)}
\newcommand{\LZp}{\ell^{p}(\mathbb Z)}
\newcommand{\m}{\mathcal {M}}

\newtheorem{thm}{Theorem}[section]
\newtheorem{cor}[thm]{Corollary}
\newtheorem{prop}[thm]{Proposition}
\newtheorem{lem}[thm]{Lemma}
\theoremstyle{remark}

\newtheorem{qu}{\bf Question}
\theoremstyle{definition}

\newtheorem{ex}[thm]{Example}
\numberwithin{equation}{section}

\title{On the direct sum of two bounded linear operators and subspace-hypercyclicity}
\bf
\author[1]{\bf\footnotesize Nareen Bamerni \thanks{nareen\_bamerni@yahoo.com}}
\author[2]{\bf Adem K{\i}l{\i}\c{c}man \thanks{akilicman@yahoo.com}}
\affil[1]{\bf Department of Mathematics, University of Duhok, Kurdistan Region, Iraq}
\affil[2]{\bf Department of Mathematics, University Putra Malaysia,
43400 UPM, Serdang, Selangor, Malaysia}

\begin{document}
\date{}
\maketitle
\begin{abstract}
In this paper, we study the relation between subspace-hypercyclicity and the direct sum of two operators. In particular, we show that if the direct sum of two operators is subspace-hypercyclic, then both operators are subspace-hypercyclic; however, the converse is true for a stronger property than subspace-hypercyclicity. Moreover, we prove that if an operator $T$ satisfies subspace-hypercyclic criterion, then $T\oplus T$ is subspace-hypercyclic. However, we show that the converse is true under certain conditions.\\

{\bf keywords:} Hypercyclic operators, Direct sums
\end{abstract}

%\thanks
  % The research has been supported by ...
%\endthanks

\section{Introduction}\label{sec1}
A bounded linear operator $T$ on a separable Banach space $\h$ is hypercyclic if there is a vector $x\in \h$ such that $Orb(T,x)=\set{T^nx:n\ge 0}$ is dense in $\h$, such a vector $x$ is called  hypercyclic for $T$. The first example of a hypercyclic operator on a Banach space was constructed by Rolewicz in 1969 \cite{rolewicz1969orbits}. He showed that, if $B$ 
is the backward shift on $\LNp$ then $\lm B$ is hypercyclic if and only if $|\lm|> 1$.\\
%The hypercyclicity concept was probabely born with the thesis of Kitai in 1982 \cite{Kitai} who showed that if $T_1 \oplus T_2$ is hypercyclic in $\h_1 \oplus \h_2$, then $T_1$ and $T_2$ are hypercyclic in $\h_1$ and $\h_2$ respectively.\\

%The studying of scaled and disk orbit of an operator was motivated by the Rolewicz example \cite{Rolewicz}. In 1974, Hilden and Wallen \cite{Hilden} defined the supercyclicity concept. An operator $T$ is called supercyclic if there is a vector $x$ such that the cone generated by $Orb(T,x)$ is dense in $\h$. Similarly, the notion of a diskcyclic operator was introduced by Zeana \cite{cyclic}. An operator $T$ is called diskcyclic if there is a vector $x \in \h$ such that the disk orbit $\D Orb(T,x)=\set{\al T^nx: \al\in \co, |\al|\le 1, n\in \N}$ is dense in $\h$, such a vector $x$ is called  diskcyclic for $T$. For more information on these operators, the reader may consult \cite{dynamic,Erdman, m3,m5}.\\

In 2011, Madore and Mart\'{i}nez-Avenda\~{n}o \cite{madore2011subspace} considered the density of the orbit in a non-trivial subspace instead of the whole space, such a concept is called subspace-hypercyclicity. An operator is subspace-hypercyclic for a subspace $\m$ of $\h$ (or $\m$-hypercyclic, for short) if there exists a vector $x\in\x$ whose orbit is dense in $\m$. Also,  Madore and Mart\'{i}nez-Avenda\~{n}o \cite{madore2011subspace} defined subspace-transitive operators and showed that every subspace-transitive is subspace-hypercyclic. Talebi and Moosapoor \cite{talebi2012subspace} defined subspace-mixing concept which is a stronger property than subspace-transitive. For more information on subspace-hypercyclicity, the reader may refer to \cite{le2011subspace, rezaei2013notes, Bamerni_2015}. \\

%In 2012, Xian-Feng et al. \cite{sub super} defined the subspace-supercyclic operator in which the scaled orbit of an operator is dense in a subspace. For more information the reader may refer to \cite{79, sub super}.\\    
In 1982, Kitai \cite{kitai1984invariant} showed that if $T_1 \oplus T_2$ is hypercyclic, then $T_1$ and $T_2$ are hypercyclic. For subspace-hypercyclicity, Madore and Mart\'{i}nez-Avenda\~{n}o showed that there exists an operator $T$ on a Banach space $\x$ such that $T\oplus I$ is subspace-hypercyclic operator for the subspace $\x\oplus \set{0}$  \cite[Example 2.2]{madore2011subspace}. However, it is clear that the identity operator $I$ cannot be subspace-hypercyclic for any nontrivial subspace. Therefore, Kitai's result cannot be extended to subspaces. On the other hand, for nontrivial subspaces and operators, we have the following question:
\begin{qu}\label{1}
If the direct sum of two operators is subspace-hypercyclic, are both operators subspace-hypercyclic?
\end{qu}

De la Rosa and Read \cite{de2009hypercyclic} showed that the converse of Kitai's result is not true by giving a hypercyclic operator $T$ such that $T\oplus T$ is not. This yields to ask the analogous question for subspace-hypercyclic operators. In particular, we have 
\begin{qu}\label{2}
If two operators are subspace-hypercyclic, is their direct sum subspace-hypercyclic?
\end{qu}

On the other hand, if $T$ satisfies hypercyclic criterion, then $T\oplus T$ is hypercyclic \cite{bayart2009dynamics}. In 1999, B\'{e}s and Peris \cite{bes1999hereditarily} proved that the converse is also true; i.e, if $T\oplus T$ is hypercyclic, then $T$ satisfies hypercyclic criterion. Now, it is natural to ask analogous questions for subspace-hypercyclic operators. In particular, we have the following questions:
\begin{qu}\label{3}
If $T$ satisfies subspace-hypercyclic criterion, is $T\oplus T$ subspace-hypercyclic?
\end{qu}

\begin{qu}\label{4}
If $T\oplus T$ is subspace-hypercyclic, does $T$ satisfy hypercyclic criterion?
\end{qu}

In the second section of this paper, we answer \Cref{1} positively; i.e, if $T_1\oplus T_2$ is subspace-hypercyclic, then $T_1$ and $T_2$ are subspace-hypercyclic. Moreover, we show that \Cref{2} has a partial positive answer. In particular, if two operators are subspace-transitive and at least one of them is subspace-mixing, then their direct sum is subspace-transitive which in turn is subspace-hypercyclic. Also, we show that if the direct sum of two operators satisfies subspace-hypercyclic criterion, then both operators do. However, we give counterexamples to show that the converse is not true in general. On the other hand, we show the converse is only true for a single operator; i.e, if $T$ satisfies subspace-hypercyclic criterion, then $T\oplus T$ does, which gives a positive answer to \Cref{3}. Furthermore, we show that under certain conditions if $T\oplus T$ is subspace-hypercyclic, then $T$ satisfies  subspace-hypercyclic criterion which answers \Cref{4} partially.\\

Let $\m_1$ and $\m_2$ be subspaces of a Banach space $\h$, then the direct sum of $\m_1$ and $\m_2$ is defined as follows 
$$\m_1 \oplus \m_2=\set{(x,y):x\in \m_1, y\in\m_2}$$ 
For more information and details on the direct sum of Banach spaces, the reader may refer to \cite{conway2013course}.  Let $\set{e_i: e_i=(\cdots,0,1,0,\cdots), i\in \Z}$ be a standard basis for the sequence space $\LZ$, then we have the following theorems.

  \begin{thm} \label{forw}\cite{bamerni2015subspace}
Let $T$ be an invertible bilateral forward weighted shift in  ${\ell^{2}(\mathbb Z)}$ with a positive weight sequence $\{w_n\}_{n\in\mathbb{Z}}$ and  ${\mathcal {M}}$ be a subspace of ${\ell^{2}(\mathbb Z)}$. Then $T$ is ${\mathcal {M}}$-transitive if and only if there exist an increasing sequence of positive integers $\{n_k\}_{k\in{\mathcal {N}}}$ and $e_{m_i}\in \{e_r\}\cap \m$ such that $T^{n_k}{\mathcal {M}} \subseteq {\mathcal {M}}$ for all $k\in {\mathcal {N}}$ and
  	\begin{equation}\label{65}  	
		\displaystyle\lim_{k\to \infty}\prod_{j=m_i}^{m_i+n_k-1}w_j=0 \;{\rm  and }\; \displaystyle\lim_{k\to \infty}\prod_{j=1+m_i}^{n_k+m_i}\frac{1}{w_{-j}}=0
		\end{equation}
\end{thm}
%-----------------------------------------------------------------------------------------------------------------
\begin{thm}\label{28}\cite{bamerni2015subspace}
Let $T_1$ and $T_2$ be invertible bilateral forward weighted shifts in  ${\ell^{2}(\mathbb Z)}$ with a positive weight sequence
$\{w_n\}_{n\in\mathbb{Z}}$ and $\{a_n\}_{n\in\mathbb{Z}}$, respectively. Let ${\mathcal {M}}_1$ and ${\mathcal {M}}_2$ be closed subspaces of ${\ell^{2}(\mathbb Z)}$. Then $T_1 \oplus T_2$ is ${\mathcal {M}}_1 \oplus {\mathcal {M}}_2$-transitive if and only if there exist $e_{m_i}\in \{e_r\}\cap \m_1$, $e_{h_p}\in \{e_r\}\cap \m_2$ and an increasing sequence of positive integers $\{n_k\}_{k\in{\mathcal {N}}}$ such that $(T_1 \oplus T_2)^{n_k}({\mathcal {M}}_1 \oplus {\mathcal {M}}_2)\subseteq {\mathcal {M}}_1 \oplus {\mathcal {M}}_2$ for all $k\in {\mathcal {N}}$ and
\begin{equation}\label{25}
  	\displaystyle\lim_{k\to \infty} \max \left\{ \prod_{j=m_i}^{m_i+n_k-1}w_j , \prod_{j=h_p}^{h_p+n_k-1}a_j  \right\}=0 
\end{equation}		
		and
\begin{equation}\label{26}		
		\displaystyle\lim_{k\to \infty} \max \left\{\prod_{j=1-m_i}^{n_k-m_i}\frac{1}{w_{-j}} ,  \prod_{j=1-h_p}^{n_k-h_p}\frac{1}{a_{-j}}\right\}=0
		\end{equation}
\end{thm}

\section{Main Results}

In this paper, all Banach spaces are infinite dimensional separable over the field $\co$ of complex numbers.  We always assume that a subspace $\m$ of a Banach space $\h$ is nontrivial ($\m \neq \h$ and $\m \neq \{0\}$) and topologically closed. Also, we assume that every bounded linear operator $T$ is nontrivial ($T\neq I$ and $T\neq 0$, where $I$ is identity operator and $0$ is zero operator) unless otherwise stated. We denote $HC(T,\m)$ the set of all $\m$-hypercyclic vectors for $T$.\\

\begin{prop}\label{chMD}
Let $T\in{\mathcal B(\mathcal H)}$, and ${\mathcal {M}}$ be a closed subspace of ${\mathcal H}$.  The following statements are equivalent:
\begin{enumerate}
\item \label{aa1} $T$ is ${\mathcal {M}}$-transitive,
  \item \label{bb1} for each $x,y\in {\mathcal {M}}$, there exist sequences $\{x_k\}_{k\in{\mathcal {N}}} \subset {\mathcal {M}}$ and an increasing sequence of positive integers $\{n_k\}_{k\in{\mathcal {N}}}$  such that  $T^{n_k}{\mathcal {M}}\subseteq {\mathcal {M}}$ for all $k\ge 1$, $x_k\to x$ and $T^{n_k}x_k\to y$ as $k\to \infty$,
  \item \label{cc1} for each $x,y\in{\mathcal {M}}$ and each $0$-neighborhood $W$ in ${\mathcal {M}}$, there exist $z\in{\mathcal {M}}$ and $n\in{\mathcal {N}}$
  such that $x-z\in W$, $T^n z-y\in W$ and $T^n{\mathcal {M}}\subseteq {\mathcal {M}}$.
\end{enumerate}
\end{prop}
\begin{proof}
$(1) \Rightarrow (2)$:
Let $x,y\in\m$. For all $k\geq 1$, suppose that
$B_k=\mathbb B(x,1/k)$ and  $B'_k=\mathbb B(y,1/k)$ are two open balls in $\h$, then  $A_k=B_k \cap \m$ and  $A'_k=B'_k \cap \m$ are relatively open subsets of $\m$. By \cite[Theorem 3.3.]{madore2011subspace},
there exist sequence $\set{n_k}$ in $\N$ and $\{x_k\}$ in $\m$ such that for all $k\geq 1$,
\[x_k \in T^{-n_k}(A'_k)\cap A_k \mbox{ and } T^{n_k}\m\subseteq \m.\]
It follows that
\[x_k\in A_k \mbox{ and } T^{n_k}\left(x_k\right)\in A'_k .\] 
Then, as $k\to \iy$ the desired result follows. \\
$(2) \Rightarrow (3)$: Follows immediately from part (2) by taking $z=x_k$ and $n_k=n$ for a large enough $k\in \N$.\\
$(3) \Rightarrow (1)$:
Let $U$ and $V$ be two nonempty open subset of $\m$. Let $W$ be a neighborhood for zero, pick $x\in U$ and $y\in V$, so there exist $z\in\m$ and $n\in\N$ such that $x-z\in W$, $T^n(z)-y\in W$ and $T^n\m\subseteq \m$. 
It follows that $U\cap T^{-n} \left(V\right)\neq \phi$ which proves (1), by \cite[Theorem 3.3.]{madore2011subspace}.
\end{proof}

The following theorem gives a positive answer to \Cref {1}. 
\begin{thm}\label{cc34}
Let $\m_1$ and $\m_2$ be closed subspaces of $\x$ and $T_1\oplus T_2$ is $(\m_1\oplus \m_2)$-hypercyclic, then $T_1$ and $T_2$ are $\m_1$-hypercyclic and $\m_2$-hypercyclic, respectively. 
\end{thm}
\proof
Let $a\in \m_1$ and $b\in\m_2$, and let $(x,y)\in HC(T_1\oplus T_2,\m_1\oplus \m_2)$, then there exist an $\eps>0$ and an increasing sequence of positive integers $\set{n_k}_{k\in\N}$ such that 
$$\left\|\brc{T_1\oplus T_2}^{n_k}(x,y)-(a,b)\right\|^2_{\m_1\oplus \m_2}\le \eps.$$ 
It follows that, 
$$\norm{T_1^{n_k}x-a}^2_{\m_1}+\norm{T_2^{n_k}y-b}^2_{\m_2}\le \eps.$$
Then 
$$\norm{T_1^{n_k}x-a}_{\m_1}\le \eps \mbox{ and }     \norm{T_2^{n_k}y-b}_{\m_2}\le \eps.$$
Thus, there exists an increasing sequence of positive integers $\set{n_k}_{k\in\N}$ such that $\set{T_1^{n_k}x:k\ge 1}$ and $\set{T_2^{n_k}y:k\ge 1}$ are dense in  $\m_1$ and $\m_2$, respectively. Therefore $Orb(T_1, x)$ and $Orb(T_2, y)$ are dense in  $\m_1$ and $\m_2$, respectively. \\
\endproof

The following two results show that the converse of \Cref{cc34} holds true under some conditions, which gives a partial positive answer to \Cref {2}. 
%===============================================================================
\begin{thm}\label{cc36}
If $T_1$ and $T_2$ are $\m_1$-transitive and $\m_2$-transitive, respectively, and at least one of them is subspace-mixing, then $T_1\oplus T_2$ is $(\m_1\oplus \m_2)$-transitive.
\end{thm}
\proof
Suppose that, without loss of generality, that $T_1$ is $\m_1$-mixing. Let $U_1\oplus V_1$ and $U_2\oplus V_2$ be open sets in $\m_1 \oplus \m_2$, then $U_1$ and $U_2$ are open in $\m_1$, and $V_1$ and $V_2$ are open in $\m_2$. By hypothesis, there exist two numbers $N_1,N_2\in \N$ such that 
 $$T_1^{-n}(U_1)\cap U_2\neq \phi \mbox{ and } T_1^n(\m_1)\subseteq \m_1 \mbox{ for all } n\ge N_1$$ 
and
 $$T_2^{-N_2}(V_1)\cap V_2\neq \phi \mbox{ and } T_2^{N_2}(\m_2)\subseteq \m_2. $$
Since $T_2$ is $\m_2$-transitive, then $\{n\in\N:T_2^{-n}V_1\cap V_2 \mbox{ and } T_2^n\m_2\subseteq \m_2\}$ is infinite \cite{talebi2013subspace}.
So, there exists $p\in\N$ such that
 $$T_1^{-p}(U_1)\cap U_2\neq \phi, T_2^{-p}(V_1)\cap V_2\neq \phi, T_1^p(\m_1)\subseteq \m_1 \mbox{ and } T_2^p(\m_2)\subseteq \m_2.$$
It follows that 
$$(T_1 \oplus T_2)^{-p}(U_1\oplus V_1)\cap (U_2\oplus V_2)\neq \phi \mbox{ and } (T_1 \oplus T_2)^p(\m_1 \oplus \m_2)\subseteq (\m_1 \oplus \m_2).$$
Thus, $T_1 \oplus T_2$ is $(\m_1 \oplus \m_2)$-transitive. \\
\endproof
%===============================================================================
\begin{prop}
Let $\m_1$ and $\m_2$ be closed subspaces of $\x$, then $T_1$ and $T_2$ are $\m_1$-mixing and $\m_2$-mixing; respectively, if and only if $T_1 \oplus T_2$ is $(\m_1 \oplus \m_2)$-mixing.
\end{prop}
\proof
 For the ``if" part, let $U_1$ and  $U_2$ be open sets in $\m_1$, and $V_1$ and $V_2$ be open sets in $\m_2$, then $U_1\oplus V_1$ and $U_2\oplus V_2$ are open in $\m_1 \oplus \m_2$. So there exists an $N\in \N$ such that  
$$(T_1 \oplus T_2)^{-n}(U_1\oplus V_1)\cap (U_2\oplus V_2)\neq \phi$$ and $$(T_1 \oplus T_2)^{n}(\m_1 \oplus \m_2)\subseteq (\m_1 \oplus \m_2)$$
 for all  $n\ge N$. Then 
  $$T^{-n}(U_1)\cap U_2\neq \phi, T^{-n}(V_1)\cap V_2\neq \phi, T^n(\m_1)\subseteq \m_1 \mbox{ and } T^n(\m_2)\subseteq \m_2. $$ 
Therefore, $T_1$ and $T_2$ are $\m_1$-mixing and $\m_2$-mixing, respectively.\\

The proof of the ``only if" part is similar to the proof of \Cref{cc36}. \\

\endproof
%===============================================================================
\begin{prop}\label{23}
If $T_1\oplus T_2$ satisfies $(\m_1\oplus \m_2)$-hypercyclic criterion, then $T_1$ and $T_2$ satisfy $\m_1$-hypercyclic criterion and $\m_2$-hypercyclic criterion, respectively. 
\end{prop}
\proof
By hypothesis, there exist two dense sets of the form $D_1 \oplus D_2$ and $D_3 \oplus D_4$ in $\m_1\oplus \m_2$ and  an increasing sequence of positive integers $\set{n_k}_{k\in\N}$ such that
\begin{enumerate}[(i)]
\item $(T_1\oplus T_2)^{n_k}(x'_1, x'_2) \to (0,0)$  for all  $(x'_1, x'_2)\in D_1 \oplus D_2.$
\item for each $(y'_1,y'_2)\in D_3 \oplus D_4$ there exists a sequence  $\set{(x_k,y_k)}_{k\in\N} \subset \m_1\oplus \m_2$ such that $(x_k,y_k)\to 0$ and $(T_1\oplus T_2)^{n_k}(x_k,y_k) \to (y'_1,y'_2).$
\end{enumerate}
Since $D_1$ and $D_3$ are dense sets in $\m_1$, and $D_2$ and $D_4$ are dense sets in $\m_2$, then it is easy to show that
\begin{enumerate}[(i)]
\item $T_1^{n_k}x'_1 \to 0$  for all  $x'_1\in D_1 $ and $ T_2^{n_k} x'_2 \to 0$  for all  $x'_2\in D_2$,
\item for each $y'_1\in D_3$ there exists a sequence  $\set{x_k}_{k\in\N} \subset \m_1$ such that $x_k\to 0$ and $T_1^{n_k}x_k \to y'_1$, and for each $y'_2\in D_4$ there exists a sequence  $\set{y_k}_{k\in\N} \subset  \m_2$ such that $y_k\to 0$ and $T_2^{n_k}y_k \to y'_2$.
\end{enumerate}
it is readily seen that $T_1$ and $T_2$ satisfy $\m_1$-hypercyclic criterion and $\m_2$-hypercyclic criterion, respectively. \\
\endproof
The following examples show that  the converse of \Cref{23} is not true in general. 
%=============================================================================
\begin{ex}\label{33}
Let $\m_1$ and $\m_2$ be two subspaces of $\x$ such that $\m_1 \oplus \m_2 \neq \x$. If for any two increasing sequences of positive integers $\set{n_k}_{k\in\N}$ and $\set{m_k}_{k\in\N}$, we have $\set{n_k}_{k\in\N}\cap \set{m_k}_{k\in\N} = \phi$ whenever $T_1$ satisfies $\m_1$-hypercyclic criterion with respect to $\set{n_k}_{k\in\N}$ and $T_2$ satisfies $\m_2$-hypercyclic criterion with respect to $\set{m_k}_{k\in\N}$. Then, $T_1 \oplus T_2$ does not satisfy $(\m_1 \oplus \m_2)$-hypercyclic criterion. 
\end{ex}

\begin{ex}\label{32}
Let $\m_1$ and $\m_2$ be closed subspaces of the Banach space $\LZp$. Choose sequences $\set{w_n}_{n\in\Z}$ and $\set{a_n}_{n\in\Z}$ of positive real numbers such that each one satisfies \Cref{forw} but not either Equation (\ref{25}) or Equation (\ref{26}) of \Cref{28}. Then, suppose that $T_1$ and $T_2$ are bilateral weighted shifts with weight sequence $\set{w_n}_{n\in\Z}$ and $\set{a_n}_{n\in\Z}$, respectively. It follows from \Cref{forw} and \Cref{28} that $T_1$ and $T_2$ satisfy $\m_1$-hypercyclic criterion and $\m_2$--hypercyclic criterion, respectively. However, $T_1 \oplus T_2$ does not satisfy $(\m_1 \oplus \m_2)$-hypercyclic criterion.
\end{ex}
%===============================================================================

%=================================================================================
The following proposition shows that the converse of \Cref{23} holds true for a single operator $T$ which gives a positive answer to \Cref{3}.
\begin{prop}\label{cc37}
If $T$ satisfies $\m$-hypercyclic criterion, then $T \oplus T$ satisfies $(\m\oplus \m)$-hypercyclic criterion.
\end{prop}
\proof
Since $T$ satisfies $\m$-hypercyclic criterion, then there exist two dense subsets $D_1$ and $D_2$ of $\m$ and an increasing sequence of positive integers $\set{n_k}_{k\in\N}$ such that all hypothesis of subspace-hypercyclic criterion are satisfied. It is easy to show that $D_1\oplus D_1$ and $D_2\oplus D_2$ are dense sets in $\m\oplus\m$. Let $(x_1,x_2)\in D_1\oplus D_1$ then $x_1, x_2\in D_1$. By hypothesis,  $T^{n_k}x_1 \to 0$ and  $T^{n_k}x_2 \to 0$. Thus, 
\begin{equation}\label{cc22}
(T\oplus T)^{n_k}(x_1,x_2) \to (0,0).
\end{equation}
Now, let $(y_1,y_2)\in D_2\oplus D_2$ then $y_1, y_2\in D_2$. Also, by hypothesis, there exist two sequences say $\set{x_k}_{k\in\N}$ and $\set{y_k}_{k\in\N}$ in $\m$ such that 
 $$x_k\to 0 \mbox{ and } T^{n_k}x_k \to y_1,$$
and
$$y_k\to 0 \mbox{ and } T^{n_k}y_k \to y_2.$$
Therefore,
\begin{equation}\label{cc23}
(x_k,y_k)\to (0,0) \mbox{ and } (T\oplus T)^{n_k}(x_k,y_k)\to (y_1,y_2).
\end{equation}
Finally, since  $T^{n_k}\m \subseteq \m$, then
 \begin{equation}\label{cc24}
(T\oplus T)^{n_k}(\m\oplus\m) \subseteq \m\oplus\m.
\end{equation}
It follows by \Cref{cc22}, \Cref{cc23} and \Cref{cc24}  that $T\oplus T$ satisfies $\m\oplus\m$-hypercyclic criterion, and thus hypercyclic. \\
\endproof

\begin{cor}\label{201}
If $T$ satisfies subspace-hypercyclic criterion, then $T\oplus T$ is subspace-hypercyclic.
\end{cor}
%================================================================================================================
For hypercyclicity, B\'{e}s and Peris \cite{bes1999hereditarily} proved that the converse of the \Cref{201} above is also true . The next theorem shows that the converse of \Cref{201} is true under certain conditions, which gives a partial positive answer to \Cref{4}. First we need the following lemma.

\begin{lem}\label{1.3}
Let $T, S\in\bx$ satisfy the equation $TS=ST$ and $\m$ be a subspace of $\x$. If $x\in HC(T,\m)$, then $Sx\in HC(T,S\m)$.
\end{lem}
\proof
Let $x\in HC(T,\m)$, then $Orb(T,x)\cap \m $ is dense in $\m$. Thus
\begin{eqnarray*}
\overline{Orb(T,Sx)\cap S\m}&=& \overline{\set{ T^nSx:n\geq 0}\cap S\m}\\
                       &=&\overline{\set{ ST^nx:n\geq 0}\cap S\m}\\
											 &=&\overline{S \set{T^nx:n\geq 0}\cap S\m }\\
											&\supseteq&  \overline{S \left(\set{ T^nx:n\geq 0}\cap \m\right)}\\
											&\supseteq& S\left( \overline{\set{ T^nx:n\geq 0}\cap \m}\right)\\
											&=& S\m.
\end{eqnarray*}
It follows that $Orb(T,Sx)\cap S\m$ is dense in ${SM}$. Thus $T$ is $S\m$-hypercyclic operator with subspace-hypercyclic vector $Sx$. \\
\endproof

%\begin{lem}
%Let $T, S\in\bx$ and $ST = TS$. Let $S\m\subseteq \m$ and $S\m$ be dense in $\m$. Then $Sx\in HC(T,\m)$ if $x\in HC(T,\m)$.
%\end{lem}
%\begin{proof} The proof follows directly from \Cref{1.3}.\end{proof}

\begin{thm}\label{202}
Let $T \oplus T$ be $(\m \oplus \m)$-hypercyclic. If there exists a subspace-hypercyclic vector $(x,y)$ such that $x$ has the following properties
\begin{enumerate}[(i)]
\item \label{cc25} $x\in \bigcap_{n\in \N}T^{n}(\m)$,
\item \label{cc26} $T^j\m \subseteq \m$ whenever $T^jx \in Orb(T,x)\cap \m$.
\end{enumerate}
Then $T$ satisfies $\m$-hypercyclic criterion. 
\end{thm}
\proof
Let $(x, y)\in HC(T \oplus T, \m \oplus \m)$ such that $x$ satisfies the stated conditions. Since $x\in HC(T, \m)$ (by \Cref{cc34}), we  show that the $\m$-hypercyclicity criterion is satisfied by the dense sets $D=D_1=D_2=Orb(T,x)\cap \m$ . Since the operator $I \oplus T^n$ commute with $T \oplus T$, then by \Cref{1.3}, we have 
\begin{equation}\label{88} 
(x, T^ny)\in HC(T \oplus T, \m \oplus T^n\m)
\end{equation}
for all $n\in \N$.
Since $y\in HC(T,\m)$ (by \Cref{cc34}), then $Orb(T,y)\cap \m$ intersects every open set in $\m$. Let $U$ be an open neighborhood of $0$ in $\m$, then there exist $u_0\in U \subseteq \m, r\in \no$ such that $u_0=T^{r}y$. It follows by \Cref{88} that 
$$(x, u_0)\in HC(T \oplus T, \m \oplus T^{r}\m).$$
By condition \textit{\ref{cc25}}, we have $(0, x) \in \m \oplus T^{r}\m$, then there exists $n_0 \in \N$ such that 
$$ (T^{n_0}x, T^{n_0}u_0 )\in Orb \left(T \oplus T, (x, u_0)\right)\cap \left(\m \oplus T^{r}\m \right)$$
 is arbitrarily close to $(0, x)$. Since $T^{n_0}x\in \m$, then $T^{n_0}x\in D$, and so $T^{n_0}\m \subseteq \m$ by condition \textit{\ref{cc26}}.
By continuing the same process, we get a sequence $u_k \to 0$ in $\m$ and $n_k\in \N$ such that  $T^{n_k}x \to 0$,  $T^{n_k}u_k \to x$ and $T^{n_k}\m \subseteq \m$.\\
Let $C=\set{i\in\N_0:T^ix\in D}$, and let us define maps $S_{n_k}:D \to \m$ by $S_{n_k}(T^ix)=T^iu_k; i\in C$ (by condition \textit{\ref{cc26}}, it is clear that $S_{n_k}(T^ix) \in \m$ for all $k\in \no, i\in C$). Now, for all $T^ix\in D$, we have
\begin{eqnarray*}
&&T^{n_k}(T^ix)=T^i(T^{n_k}x) \to 0,\\
&&S_{n_k}(T^ix)=T^iu_k \to 0 \mbox{ since } u_k\to 0,\\
&&T^{n_k}S_{n_k}(T^ix)=T^iT^{n_k}u_k \to T^ix.
\end{eqnarray*}
for all $i\in C$. It follows that $T$ satisfies $\m$-hypercyclic criterion. \\
\endproof
It follows by \Cref{202} above that hypercyclicity and subspace-hypercyclicity do not share the same properties. Therefore, one may wonder whether there are some more similarities and differences between hypercyclicity and subspace-hypercyclicity.

\end{document}